\documentclass{amsart}

\usepackage{amssymb}

\newtheorem{theorem}{Theorem}[section]
\newtheorem{lemma}[theorem]{Lemma}
\newtheorem{proposition}[theorem]{Proposition}

\newtheorem{remar}[theorem]{Remark}

\theoremstyle{definition}

\newenvironment{remark}{\begin{remar}\rm}{\end{remar}}

\numberwithin{equation}{section}

\newenvironment{axiom}{\begin{list}{$\bullet$}{\setlength{\labelsep}{.3cm}%
\setlength{\leftmargin}{1.4cm}\setlength{\rightmargin}{0cm}%
\setlength{\labelwidth}{1cm}\setlength{\itemsep}{0pt}}}{\end{list}}
\newcommand{\ax}[1]{\item[{\bf #1}\hfill]\index{#1}}

\newcommand{\bfind}[1]{{\bf #1}}
\newcommand{\n}{\par\noindent}

\newcommand{\sn}{\par\smallskip\noindent}
\newcommand{\mn}{\par\medskip\noindent}

\newcommand{\pars}{\par\smallskip}
\newcommand{\parm}{\par\medskip}
\newcommand{\parb}{\par\bigskip}

\newcommand{\sep}{^{\rm sep}}
\newcommand{\chara}{\mbox{\rm char}\,}
\newcommand{\trdeg}{\mbox{\rm trdeg}\,}

\newcommand{\cal}{\mathcal}
\newcommand{\N}{\mathbb{N}}

\newcommand{\Z}{\mathbb{Z}}

\newcommand{\cT}{\mathcal{T}}

\begin{document}

\title[Henselian Rationality]{Elimination of Ramification II: Henselian Rationality}

\author{Franz-Viktor Kuhlmann}
\address{Institute of Mathematics, University of Szczecin, ul. Wielkopolska 15, 	  	  	
70-451 Szczecin, Poland}
\email{fvk@usz.edu.pl}
\thanks{I thank the referee for his very careful reading of the article and his many useful corrections and
suggestions. Further, I wish to thank Peter Roquette for his invaluable help and
support. I also thank F.~Delon, B.~Green, H.~Knaf, F.~Pop and A.~Prestel for
inspiring discussions and suggestions.\\
I am presently supported by a Polish Opus grant 2017/25/B/ST1/01815. During the earlier work on this paper,
I was partially supported by a Canadian NSERC grant and held
visiting professor positions at the University of Silesia in Katowice and the University of Szczecin,
Poland. I wish to thank the faculties of mathematics of these universities for their hospitality.
}

\date{25. 1. 2019}

\begin{abstract}
We prove in arbitrary characteristic that an immediate valued algebraic function
field $F$ of transcendence degree 1 over a tame field $K$ is contained
in the henselization of $K(x)$ for a suitably chosen $x\in F$. This
eliminates ramification in such valued function fields. We give
generalizations of this result, relaxing the assumption on $K$. Our
theorems have important applications to local uniformization and to the
model theory of valued fields in positive and mixed characteristic.
\end{abstract}

\maketitle

%
%
\section{Introduction}
%
%
%
\subsection{The Main Theorem}
In this paper, we prove a structure theorem for a special sort of valued
function fields, which complements our ``Generalized Stability Theorem''
proved in \cite{[K8]} and has important applications to local uniformization and
the model theory of valued fields, which we will discuss below. By
``function field'' we will always mean ``algebraic function field''.

By $(K,v)$ we denote a field $K$ equipped with a (Krull) valuation $v$.
We write a valuation in the classical additive way, that is, the value
group, denoted by $vK$, is an additively written ordered abelian group,
and the ultrametric triangle law reads as $v(a+b)\geq\min \{va,vb\}$.
We denote the valuation ring by ${\cal O}_K$ and its maximal ideal by
${\cal M}_K\,$, the residue field by $Kv$, by $va$ the value of an element $a\in K$, and by $av$ its residue.
When we talk of a valued field extension $(L|K,v)$ we mean that $(L,v)$
is a valued field, $L|K$ a field extension, and $K$ is endowed with the restriction of $v$.
An extension $(L|K,v)$ of valued fields is called \bfind{immediate} if
the canonical embeddings $vK\hookrightarrow vL$ and $Kv\hookrightarrow Lv$ are onto. A valued field $(K,v)$
is \bfind{henselian} if it satisfies Hensel's Lemma, or equivalently, if the extension of $v$
to the algebraic closure $\tilde{K}$ of $K$ is unique. A \bfind{henselization} of $(K,v)$ is a minimal henselian
extension of $(K,v)$, in the sense that it admits a unique valuation preserving embedding over $K$ in
every other henselian extension of $(K,v)$. In particular, if $w$ is any extension of $v$
to $\tilde{K}$, then $(K,v)$ has a unique henselization in
$(\tilde{K},w)$. Henselizations of $(K,v)$ are unique up to valuation preserving
isomorphism over $K$; therefore, we will always speak of {\it the} henselization of $(K,v)$, and denote it by
$(K^h,v)$. Note that $(K^h|K,v)$ is always an immediate separable-algebraic extension.

Throughout the paper, when dealing with a valued function field $(F|K,v)$ we will assume that $v$ is extended to
$\tilde{F}$, and this extension will again be denoted by $v$. Then all subfields of
$(\tilde{F},v)$ have a unique henselization within $(\tilde{F},v)$.

\pars
An algebraic extension $(L|K,v)$ of a henselian field $(K,v)$ is called \bfind{tame} if every finite subextension
$E|K$ of $L|K$ satisfies the following conditions:
\begin{axiom}
\ax{(TE1)} The ramification index $(vE:vK)$ is not divisible by $\chara Kv$,
\ax{(TE2)} The residue field extension $Ev|Kv$ is separable,
\ax{(TE3)} The extension $(E|K,v)$ is \bfind{defectless}, i.e.,
$[E:K]=(vE:vK)[Ev:Kv]$.
\end{axiom}
\begin{remark}
This notion of ``tame extension'' does not coincide with the notion of
``tamely ramified extension'' as defined in \cite{[End]}, page 180.
The latter definition requires (TE1) and (TE2), but not (TE3). Our tame
extensions are the defectless tamely ramified extensions in the sense of
\cite{[End]}. In particular, in our terminology, proper immediate algebraic
extensions of henselian fields are not called tame, because they cause a
lot of problems for local uniformization and in the model theory of
valued fields.
\end{remark}

For later use we note that in the situation of the above definition, the Lemma of Ostrowski (cf.\
\cite[Th\'eor\`eme 2, p.~236]{[R]} or \cite[Corollary to Theorem~25, Section G, p.~78]{[Z--S]}) states that
the quotient $[E:K]/(vE:vK)[Ev:Kv]$ is a power of the residue characteristic $\chara Kv$ if this is positive,
and equal to $1$ otherwise.

\pars
A \bfind{tame field} is a henselian field for which all algebraic
extensions are tame. Likewise, a \bfind{separably tame valued field} is a
henselian field for which all separable-algebraic extensions are tame.
All henselian fields of residue characteristic $0$ and all algebraically
maximal Kaplansky fields are tame fields (but not every tame field is a Kaplansky field); see \cite{[K7]}
for details. A valued field is called \bfind{algebraically maximal} (or \bfind{separable-algebraically
maximal}) if it does not admit proper immediate algebraic (or separable-algebraic, respectively) extensions.
Since the henselization of a valued field is an immediate separable-algebraic extension, it follows that every
separable-algebraically maximal valued field is henselian.

Take a henselian field $(K,v)$ and extend $v$ to $\tilde{K}$. Denote by $K\sep$ the separable-algebraic
closure of $K$. The \bfind{absolute
ramification field of $(K,v)$} is the ramification field of the normal extension $(K\sep|K,v)$. It is the unique
maximal tame extension of $(K,v)$ by \cite[Theorem (22.7)]{[End]} (see also \cite[Proposition 4.1]{[K--P--R]}).
Hence a henselian field is tame if and only if its absolute ramification field is already algebraically closed;
in particular, every tame field is perfect. Likewise, a henselian field is separably tame
if and only if its absolute ramification field is separable-algebraically closed. Further, every tame field is
algebraically maximal and every separably tame field is separable-algebraically maximal because by (TE3),
$(vE:vK)=[Ev:Kv]=1$ implies $[E:K]=1$.

\pars
An extension $(F|K,v)$ of valued fields will be called \bfind{henselian
rational} if it admits a transcendence basis $\cT\subset F$ such that
$F$ lies in the henselization of the rational function field $K(\cT)$, or in other words, any henselization of
$(F,v)$ is also a henselization of $(K(\cT),v)$. The basic version of our main theorem is:
\begin{theorem}                \label{stt3}
Let $(K,v)$ be a tame field and $(F|K,v)$ an immediate function field. If
its transcendence degree is 1, then $(F|K,v)$ is henselian rational. In the general case of transcendence
degree $\geq 1$, given any immediate extension $(N,v)$ of $(F,v)$ which is a tame field, there is a finite immediate
extension $(F_1,v)$ of $(F,v)$ within $(N,v)$ such that $(F_1|K,v)$ is henselian rational.
\end{theorem}

Actually, we will prove more general results:
\begin{theorem}                \label{stt3'}
Let $(K,v)$ be a separably tame field and $(F|K,v)$ an immediate function field, with $F|K$ a separable extension.
If its transcendence degree over $K$ is 1, then $(F|K,v)$ is henselian rational. In the
general case of transcendence degree $\geq 1$, given any immediate separable extension $(N,v)$ of $(F,v)$
which is a separably tame field, there is a finite immediate separable extension $(F_1,v)$ of $(F,v)$ within
$(N,v)$ such that $(F_1|K,v)$ is henselian rational.
\end{theorem}
\n
In Section~\ref{sectprstt3} we will show that the latter theorem implies the former. For $\chara K=0$
both theorems coincide.
To avoid case distinctions, we will often work with the formulation given in the
latter theorem, even when $\chara K=0$.
%

\pars
Take an immediate function field $(F|K,v)$ of transcendence degree $1$. Choose
any $x\in F$ transcendental over $K$. Then the finite extension $F^h|K(x)^h$ is immediate
(cf.\ Lemma~\ref{FhFEh} below). If $\chara Kv=0$,
then the Lemma of Ostrowski shows that this extension is trivial. This proves Theorem~\ref{stt3'}\label{+disc}
in the case of $\chara Kv=0$; for details and a more general result, see Theorem~\ref{rc0friffhr}. Except for that
theorem, we will always assume that $\chara Kv=p>0$.

\pars
In order to prove Theorem~\ref{stt3'}, we will reduce to the case of $(K,v)$ having rank 1, i.e., its value
group being archimedean ordered. We will prove that under this assumption, Theorem~\ref{stt3'} holds whenever
$K$ is separable-algebraically closed (Proposition~\ref{,stt1}). Then the following theorem will prove the rank 1
case of the first assertion of Theorem~\ref{stt3'}, because if $(K,v)$ is a separably tame field, then it is
separable-algebraically maximal and the extension $K\sep|K$ is tame by definition.

If $F$ and $L$ are subfields of a common extension field $E$, then we define the \bfind{compositum} $F.L$ to be
the smallest subfield of $E$ that contains both $F$ and $L$. Further, we denote the completion of a valued field
$(K,v)$ by $K^c$.
\begin{theorem}                   \label{,pdp}
Let \mbox{$(K,v)$} be a separable-algebraically maximal field of rank 1, and let $(F|K,v)$ be a separable
immediate function field of transcendence degree 1. Assume that there is no valuation preserving embedding of $F$
in $K^c$ over $K$. If $(F.L|L,v)$ is a henselian
rational function field over $L$ for some tame extension $(L|K,v)$, then also $(F|K,v)$ is henselian rational.
\end{theorem}
We will deduce this theorem in Section~\ref{sectred2} from \cite[Theorem 14.5]{[K--V]}.
\pars
In order to prove Theorem~\ref{stt3'} under the additional assumptions, we reduce further to the analysis
of Galois extensions $E|K(x)^h$ of degree $p$ (see the more detailed discussion of our methods below). We will find
some $y\in E$ such that $E=K(y)^h$; then \cite[Theorem 11.1]{[K--V]} shows us that $y$ can already be chosen in $F$.

\pars
Let us also note:
\begin{proposition}                            \label{N}
Extensions $(N,v)$ of $(F,v)$ as in Theorem~\ref{stt3} or Theorem~\ref{stt3'} always exist.
\end{proposition}

\pars
In Section~\ref{sectvalg} we will deduce the following theorem from Theorem~\ref{stt3}:
\begin{theorem}                \label{stt4}
Let $(F|K,v)$ be a valued function field of transcendence degree 1 such that $vF/vK$ is a torsion group and $Fv|Kv$
is algebraic. Extend $v$ to the algebraic closure of $F$. Then there is a finite extension $L$ of $K$ such that
$(F.L|L,v)$ is henselian rational.
\end{theorem}

\parb
The core methods for the proof of Theorems~\ref{stt3} and~\ref{stt3'} were developed in \cite{[K1]}. Later we found
out that they are very similar to an approach put forward by S.~Abhyankar in \cite{[A1]}: ramification theory (i.e.,
the fact that ramification groups are $p$-groups) is used to reduce the proofs to the central problem of
dealing with Galois extensions of degree $p=\chara Kv$. In the present paper, this reduction is based on
Lemma~\ref{,mcc} below. In the \bfind{equal characteristic case} where $\chara K=\chara Kv$, such an extension is
generated by an element whose minimal polynomial is an \bfind{Artin-Schreier polynomial} $X^p-X-c$. The desired
results in \cite{[A1]} as well as in
the present paper are then achieved by finding a suitable normal form for $c$ (cf.\ Section~\ref{sectnf}); here the
additivity of the polynomial $X^p-X$ plays a crucial role. Note that Abhyankar deals with
polynomials of the form $X^p-d^{p-1}X-c$ since he works over rings; as we work over fields we have the benefit of
using the original Artin-Schreier polynomial. In the \bfind{mixed characteristic case} where $\chara K=0$
and $\chara Kv=p>0$, the cyclic extensions are generated by $p$-th roots, assuming that the fields in question
contain a primitive $p$-th root of unity. Also for this case we derive suitable normal forms. It is worth mentioning
that the tools for this purpose are developed in Section~2.2 of \cite{[K8]} by transforming a polynomial $X^p-c$ into
one that is Artin-Schreier modulo coefficients of higher value. In this way, we can use a form of additivity modulo
terms of higher value.

Abhyankar pulls up local uniformization through cyclic extensions of degree $p$ (``going up'' -- see
\cite[Theorem 4]{[A1]}), and in the present paper we do the same for henselian rationality. In the case of Abhyankar
places on function fields, the same is done in \cite{[K8]} for the property of being a defectless field. The
remaining case of ``going up'' for degrees prime
to $p$ (cf.\ \cite[Theorem 2 and \$5]{[A1]}) is handled without breaking them up into extensions of prime degree by
making use of the properties of tame extensions. The same is true for our analogue of Abhyankar's ``coming down''
(cf.\ \cite[Theorem 2 and \$6]{[A1]}): henselian rationality (as well as the property of being a defectless field)
can be pulled down through every tame extension (Theorem~\ref{,pdp}).

Recently, inspired by our approach laid out in \cite{[K--K1],[K3]}, V.~Cossart and O.~Piltant have used the same
reduction procedure in \cite{CP1,CP2} to prove resolution of singularities for threefolds (see their remark about
their Theorem~8.1 in the Introduction of \cite{CP1}). The problem is reduced to dealing with Artin-Schreier
extensions and
purely inseparable extensions of degree $p$. Note that the latter can be avoided in the present paper by using the
fact that the function fields we consider are separably generated.

\pars
When our attention was drawn to H.~Epp's paper \cite{[Ep]} we realized that our methods in dealing with
Artin-Schreier extensions in \cite{[K1],[K8]} and in the present paper constitute a generalization of some
of his methods. Based on our own experience with the pitfalls of the mentioned deduction of normal forms, we noticed
a gap in one of his proofs, which we filled in \cite{[K4]}. In turn, a gap in one of our proofs in \cite{[K1]} was
later filled by Yu.~Ershov; cf.\ Remark~\ref{remershov} below.

%
%
\subsection{Applications}
\mbox{ }\sn
$\bullet$ \ {\bf Elimination of ramification.} \
This is the task of finding a transcendence basis $\mathcal{T}$
for a given valued function field $(F|K,v)$ such that the extension $(F^h|
K(\mathcal{T})^h ,v)$ of respective henselizations is
\bfind{unramified}, that is, the residue fields form a separable
extension $Fv| K(\mathcal{T})v$ of degree equal to
$[F^h:K(\mathcal{T})^h]$, and $vF=vK(\mathcal{T})$. (Recall that passing
to the henselization does not change value group and residue field.)

Theorem~\ref{stt3} and~\ref{stt3'} show that immediate function fields
of transcendence degree 1 under the given assumptions admit elimination
of ramification. Theorem~\ref{stt4} shows that valued function fields of transcendence
degree 1 that are valuation algebraic extensions in the sense of \cite{KTrans} admit elimination
of ramification over a finite extension of the base field.

\mn
$\bullet$ \ {\bf Local uniformization in positive and in mixed
characteristic.} \ Theorem~\ref{stt3} is a crucial ingredient for our
proof that all places of algebraic function fields admit local
uniformization after a finite extension of the function field (\cite{[K--K2]}).
The analogous arithmetic case (also treated in \cite{[K--K2]}) uses
Theorem~\ref{stt3} in mixed characteristic. The proofs use solely
valuation theory.

\mn
$\bullet$ \ {\bf Model theory of valued fields.} \
In \cite{[K7]} we use Theorems~\ref{stt3} and~\ref{stt3'} to prove the following Ax--Kochen--Ershov Principle:
\begin{theorem}                             \label{tfee}
Take two tame valued fields $(K,v)$ and $(L,v)$ of positive characteristic. If $vK$ is elementarily
equivalent to $vL$ as ordered groups and $Kv$ is elementarily equivalent
to $Lv$ as fields, then $(K,v)$ is elementarily equivalent to $(L,v)$ as
valued fields.
\end{theorem}
\n
In the same paper and in \cite{[KPal]}, Theorems~\ref{stt3} and~\ref{stt3'} are also used to
prove other Ax--Kochen--Ershov Principles (which then also hold in mixed characteristic),
and further model theoretic results for tame and separably tame valued fields. The reader should note that
in the present paper we will make extensive use of the valuation theoretical preliminaries and the general
algebraic theory of tame and separably tame fields presented in Sections 2 and 3 of \cite{[K7]}.
Theorems~\ref{stt3} and~\ref{stt3'} are stated in the Introduction, but only applied in Section 7 of \cite{[K7]}
to prove model theoretic results on tame and separably tame fields.

%
%
\section{Two special cases}
We start with a lemma that we will need here as well as later in the paper.
\begin{lemma}                       \label{FhFEh}
Take an arbitrary algebraic extension $(F|E,v)$ and extend $v$ to $\tilde{F}$. Taking the respective henselizations in
$(\tilde{F},v)$, we have that $F^h=F.E^h$. Hence if $F|E$ is finite, algebraic or separable, then $F^h|E^h$ is
finite, algebraic or separable, respectively. Further, $(F|E,v)$ is immediate if and only it $(F^h|E^h,v)$ is.
\end{lemma}
\begin{proof}
As an algebraic extension of the henselian field $(E^h,v)$, also $(F.E^h,v)$ is henselian. It also contains $(F,v)$, so
it must contain $(F^h,v)$. On the other hand, $F^h$ contains $F$ and $E$, and must also contain the henselization
$E^h$. So $F.E^h\subseteq F^h$ and equality holds.

The second assertion is a direct consequence of the first. For the third assertion, just observe that
$vF^h=vF$, $F^hv=Fv$, $vE^h=vE$ and $E^hv=Ev$.
\end{proof}

A valued field $(K,v)$ is called \bfind{finitely ramified} if there is a
prime $p$ such that $vp>0$ and $vK$ has only finitely many elements between $0$ and
$vp$. In this case, $\chara K=0$ and $\chara Kv=p$. Every henselian finitely ramified field is a \bfind{defectless
field}, i.e., all of its finite extensions are defectless (\cite{[K1]}; cf.\ \cite{[KB]}).
\begin{theorem}                               \label{rc0friffhr}
Let $(K,v)$ be a valued field of residue characteristic 0 or a finitely
ramified field. Then every immediate function field over $(K,v)$ is
henselian rational.
\end{theorem}
\begin{proof}
Let $(F|K,v)$ be an immediate function field. Let ${\cal T}$ be an arbitrary transcendence basis of $F|K$. Then
also $(K(\cT)|K,v)$ is immediate, as $vK\subseteq vK(\cT)\subseteq vF=vK$ and $Kv\subseteq K(\cT)v\subseteq Fv=Kv$.
Hence by Lemma~\ref{FhFEh}, $(F^h|K({\cal T})^h,v)$ is an immediate algebraic extension. If the
residue characteristic of $(K,v)$ is 0, then the same holds for $(K({\cal T})^h,v)$, so the Lemma of
Ostrowski yields that the extension $(F^h|K({\cal T})^h, v)$ must be trivial, whence $F\subseteq K({\cal T})^h$.

If $(K,v)$ is a finitely ramified field, then so is every immediate
extension of $(K,v)$. Hence $(K({\cal T})^h,v)$ is a defectless field,
and it again follows that $F\subseteq K({\cal T})^h$.
\end{proof}

For the next theorem, note that the completion $K^c$ of a henselian field $(K,v)$ is again henselian (cf.\
\cite[Theorem~32.19]{[W]}). Hence henselizations of any subfields of this completion can be taken inside of it.
\begin{theorem}               \label{Xicr}
Let $(K,v)$ be a henselian field of arbitrary characteristic. If the
valued function field $(F|K,v)$ is a separable subextension of the
extension $K^c|K$, then $(F|K,v)$ is henselian rational; more precisely, $F\subset K(\cT)^h$
for every separating transcendence basis $\cT$ of $F|K$.
\end{theorem}
\begin{proof}
Let $\cT$ be a separating transcendence basis of $F|K$. Then $F.K(\cT)^h|K(\cT)^h$ is a separable-algebraic
subextension of $K^c|K(\cT)^h$. But this extension must be trivial since a henselian field is
separable-algebraically closed in its completion (cf.\ \cite[Theorem 32.19]{[W]}).
\end{proof}

%
%
\section{Valuation theoretical tools}
We will develop here some tools that we will later use in the proof of Lemma~\ref{,nff}.
We take an arbitrary valued field $(L,v)$ of characteristic 0 with residue characteristic $p>0$.
The following lemma has been proved in \cite[Corollary~2.11]{[K8]}:
\begin{lemma}               \label{1+y}
Let $(K,v)$ be a henselian field containing all $p$-th roots of unity.
Take any $1$-units $1+b$ and $1+c$ in $K$ (i.e., $vb>0$ and $vc>0$). Then:
\sn
a) \ $1 + b \in (1 + b + c)\cdot (K^{\times})^p\;\;$ if $\;vc > \frac{p}{p-1}vp\,$.
\sn
b) \ $1 + b \in (1 + b + c)\cdot (K^{\times})^p\;\;$ if $\;1+c\in (K^{\times})^p$ and $vbc > \frac{p}{p-1}vp\,$.
\sn
c) \ $1 + b -pc \in (1 + b + c^p)\cdot (K^{\times})^p\;\;$ if $\;vb\geq \frac{1}{p-1}vp\;$ and $\;vc^p > vp\,$.
\end{lemma}

Part c) of this lemma will play an important role in the proof of Lemma~\ref{,nff}, which is dealing with
valued fields of mixed characteristic. There we will use it to replace elements $a$ by expressions of
the form $-pa^{1/p}$. But as we will be working in a field of characteristic $0$ which contains the $p$-th roots of
unity, the expression $a^{1/p}$ does not designate a unique element (unless $a=0$). This, however, does not
matter for our purposes, as in part c) of the above lemma, $c$ can be replaced by $\zeta c$ for any $p$-th root
of unity $\zeta$. So when we use these expressions, we actually mean to say: ``choose any $p$-th root''. In the
same way we will use an operator $\Delta$ as follows:
\begin{equation}
\Delta(a):=-pa^{1/p} \>.
\end{equation}
Its inverse $\Delta^{-1}$ is in fact a function, sending $d$ to $(-d/p)^p$.

In order to track the change of value in passing from the term $a$ to the term $-pa^{1/p}$, we will use a
tool that was introduced in \cite{[Er]}. Note that $va^{1/p}=va/p$, no matter which $p$-th root $a^{1/p}$ we
have chosen. To avoid unnecessary technical complications, we will assume that
the value group $vL$ is $p$-divisible. We define a function $\delta$ on it as follows:
\[
\delta(\gamma)\>:=\> vp\,+\,\frac{1}{p}\gamma\>.
\]
We then have:
\begin{equation}
v\Delta(a) \>=\> v(-pa^{1/p}) \>=\> \delta(va) \;\mbox{ and }\; v\Delta^{-1}(d) \>=\> v(-d/p)^p \>=\>
\delta^{-1}(vd)\>.
\end{equation}

If we set $d=-pc$, then $c^p=(-d/p)^p$ and the condition ``$vc^p>vp$'' becomes
``$vd>vp+\frac{1}{p}vp=\delta(vp)$''. So part c) of Lemma~\ref{1+y} can be reformulated as:
\begin{equation}                  \label{Delta^-1}
(1 + b + d)\cdot K^p \>=\, 
(1 + b + \Delta^{-1}(d))\cdot K^p\;\; \mbox{ if $\;vb\geq
\frac{1}{p-1}vp\;$ and $\; \delta^{-1}(vd) > vp\,$.}
\end{equation}

\parm
We denote by $\Delta^i$ the $i$-th iteration of $\Delta$, and by $\delta^i$ the $i$-th iteration of $\delta$.

\begin{lemma}
The function $\delta$ has the following properties:
\sn
a) \ $\delta$ is order preserving, and for each integer $i\geq 0$, it
induces a bijection from the interval $\left[\delta^i(0),\frac{p}{p-1}vp
\right]$ onto the subinterval $\left[\delta^{i+1}(0),\frac{p}{p-1}vp
\right]$,
\sn
b) \ $\delta$ is strictly increasing on $\left[0,\frac{p}{p-1}vp
\right)$ and has a fixed point in $\frac{p}{p-1}vp$,
\sn
c) \ for each $i\geq 1$, 
\begin{eqnarray*}
& \Delta^i(a)\>=\>\Delta^i(1)\cdot a^{1/p^i} \>\mbox{ and \ }\; \Delta^{-i}(a)\>=\>\Delta^{-i}(1)\cdot a^{p^i} & \\
& \delta^i(\gamma)\>=\>\frac{1+p+\ldots+p^{i-1}}{p^{i-1}}vp \,+\, \frac{1}{p^i}
\gamma \>=\>\delta^i(0) \,+\,\frac{1}{p^i}\gamma  \;\mbox{ and \ }\>
\delta^{-i}(\gamma)\>=\>\delta^{-i}(0) \,+\,p^i\gamma\>. &
\end{eqnarray*}
\end{lemma}
\begin{proof}
If $\alpha<\beta$, then $\delta(\alpha)=vp+\frac{1}{p}\alpha< vp+
\frac{1}{p}\beta=\delta(\beta)$; so $\delta$ preserves $<$. Further,
$\gamma\geq \delta(\gamma)=vp+\frac{1}{p}\gamma$ implies that
$\frac{p-1}{p}\gamma\geq vp$, i.e., $\gamma\geq \frac{p}{p-1}vp$; so
$\delta$ is strictly increasing everywhere below $\frac{p}{p-1}vp$.
Replacing $\geq$ by $=$ in this arguments, we see that $\delta
(\frac{p}{p-1}vp)= \frac{p}{p-1}vp$. It follows that for each $i\geq 0$,
$\delta$ induces a bijection from $\left[\delta^i(0),\frac{p}{p-1}vp
\right]$ onto the subinterval $\left[\delta^{i+1}(0),\frac{p}{p-1}vp
\right]$. We have proved assertions a) and b).

Assertion c) is easily proved by induction on $i$.
\end{proof}

%
%
\section{Galois extensions of degree $p$ of $K(x)^h$}               \label{sectGalext}
Throughout this section, {\it we will assume that $(K(x)|K,v)$ is an immediate transcendental extension.} We will
investigate the structure of Galois extensions $E$ of degree $p=\chara Kv>0$ of $K(x)^h$.

If $\chara K=p>0$, then $E|K(x)^h$ is an \bfind{Artin-Schreier extension}, that is, it is
generated by an element $\vartheta\in E$ which satisfies
\begin{equation}                \label{ASext}
\vartheta^p-\vartheta\>=\>a\in K(x)^h
\end{equation}
(cf.\ \cite[VI, \$6, Theorem~6.4]{[L]}). We set $\wp(X)=X^p-X$ and observe that this is an additive polynomial,
i.e., $\wp(b+c)=\wp(b)+\wp(c)$ holds in each field of characteristic $p$. From this fact it follows that $a$ can be
replaced by any other element in $a+\wp(K(x)^h)$ without changing the extension. From Hensel's Lemma it
follows that ${\mathcal M}_{K(x)^h}
\subseteq \wp(K(x)^h)$ and therefore, $a$ can be replaced by any other element in $a+{\mathcal M}_{K(x)^h}$
without changing the extension (see the discussion at the start of Section~4 in \cite{[K8]}).

\pars
If $\chara K=0$ and $K$ contains the $p$-th roots of unity, then $E|K(x)^h$ is
generated by an element $\eta\in E$ which satisfies
\begin{equation}                \label{Kumext}
\eta^p\>=\>a\in K(x)^h
\end{equation}
(cf.\ \cite[VI, \$6, Theorem~6.2]{[L]}). Here, $a$ can be replaced by any other nonzero element in
$a\cdot(K(x)^h)^p$.

\parm
If we assume in addition that the rank of $(K,v)$ is 1, then we can say even more about the element $a$. To this end
we need the following result, which is Lemma~10.1 of \cite{[K--V]}:
\begin{lemma}                               \label{,1id}
If the rank of $(K,v)$ is 1 and $(K(x)|K,v)$ is immediate, then $K[x]$ is dense in $K(x)^h$.
\end{lemma}

Assume that the rank of $(K,v)$ is 1 and that $\chara K=p>0$. By Lemma~\ref{,1id}, for every $a\in K(x)^h$ there
is $f(x)\in K[x]$ such that $a-f(x)\in {\mathcal M}_{K(x)^h}\,$. Hence in (\ref{ASext}), $a$ can be replaced by
$f(x)$, so that we have:
\begin{equation}                \label{ASextf}
E\>=\> K(x)^h(\vartheta)\;\mbox{ with }\; \vartheta^p-\vartheta\>=\>f(x)\in K(x)\>.
\end{equation}

Assume now that the rank of $(K,v)$ is 1 and that $\chara K=0$. Assume in addition that $K$ is closed under $p$-th
roots. Since $a$ lies in an immediate extension of $(K,v)$, we know that $va\in vK$, so there is some $d_1\in K$
such that $vd_1^p=-va$ and therefore, $vd_1^pa=0$. For the same reason, $d_1^pav\in Kv$ and there is some
$d_2\in K$ such that $(d_1^pd_2^pa)v=1$.
We set $d=d_1 d_2\in K$ to obtain that $v((d\eta)^p-1)>0$ and that $d\eta$ generates the $E | K(x)^h$
with $(d\eta)^p \>=\> 1+a' \in K(x)^h$ where $va'>0$. By Lemma~\ref{,1id} there is a polynomial $f(x)\in K[x]$ such
that $v(a'-f(x))> \frac{p}{p-1}vp$. Note that this implies that $vf(x)>0$, i.e., $1+f(x)$ is a 1-unit.
Hence by part a) of Lemma~\ref{1+y}, any root of the polynomial $X^p - (1+f(x))$
will also generate the extension $E|K(x)^h$. So we can assume from the start:
\begin{equation}                            \label{f(x)}
E\>=\> K(x)^h(\eta)\;\mbox{ with }\; \eta^p \>=\> 1+f(x)\in K(x) \mbox{ a 1-unit.}  
\end{equation}

\parm
We will now first determine suitable normal forms for $f(x)$ in (\ref{ASextf}) and (\ref{f(x)}), depending on
the characteristic of $K$.

Since the extension $(K(x)|K,v)$ is immediate and $x\notin K$, the set
\[
v(x-K) \>:=\> \{v(x-c)\mid c\in K\}    
\]
does not have a largest element; this follows from \cite[Theorem 1]{[Ka]}.
We say that the \bfind{approximation type of $x$ over $K$ is transcendental} if for every polynomial $h(X)\in
K[X]$ there is some $\alpha\in v(x-K)$ such that for all $c\in K$ with $v(x-c)\geq \alpha$
the value $vh(c)$ is fixed.

%
%
\subsection{Normal forms for polynomials in K[x]}               \label{sectnf}
In this section we will consider an immediate transcendental extension $(K(x)|K,v)$ with $\chara Kv=p>0$
and assume that the approximation type of $x$ over $K$ is transcendental.
\begin{lemma}                    \label{nff}
Assume that $\chara K>0$.
Then for every $f(x)\in K[x]$ there exists a finite purely inseparable extension $K'|K$ and a polynomial
\begin{equation}                     \label{nfH}
g(z)\,\in\, f(x) + \wp(K'(x)^h)
\end{equation}
satisfying:
\begin{equation}                     \label{nfh}
\left\{\; \begin{minipage}{0.85\textwidth}\noindent
$g(z) = a_n z^n + \ldots  + a_1 z + a_0 \in K'[z]\,, \mbox{ where}$\n
$z = (x - c)/d\,,\mbox{ with } vz = 0\,,\;
c\in K \mbox{ and }\, 0\not= d \in K$,\n
for all $i>0$, we have: $a_i=0$ or $va_i<0$, and $p|i \>\Rightarrow\> a_i = 0$, \n
the values $va_i$ of all nonzero $a_i\,$, $i>0$, are distinct.
\end{minipage}
\right.
\end{equation}
Note that $K[x]=K[z]$. If $(K,v)$ is perfect or separably tame, we may assume that $K'=K$.
\end{lemma}
\begin{proof}
Set $\deg f = n$. We consider the following Taylor expansion with variables $X$ and $X_0$:
\begin{equation}           \label{TE}
f(X) = \sum_{i=0}^n f_i(X_0)(X - X_0)^i
\end{equation}
where $f_i$ denotes the $i$-th Hasse-Schmidt derivative of $f$. For any $i$
which is divisible by $p$, say $i = p j$, the summand $f_i(X_0)(X - X_0)^i$ in $f(X)$ is equivalent to
\[
f_i(X_0)^{p^{-1}} (X -X_0)^j
\]
modulo $\wp(K_1[X,X_0])$, where
\[
K_1 \>=\> K\left(f_i(X_0)^{p^{-1}}\right)\>.
\]
By a repeated application of this procedure we find that modulo $\wp(K'[X,X_0])$, with $K'|K(X_0)$ a finite purely
inseparable extension, $f(X)$ is equivalent to a polynomial
\begin{equation}              \label{eqp}
f(X_0) + \sum_j{}' \left(f_j(X_0) +\sum_{\nu}{}^{(j)} f_{jp^{\nu}}(X_0)^{p^{-\nu}}\right)(X - X_0)^j\>,
\end{equation}
where:
\mn
$\bullet$ \ $\displaystyle\sum_j{}'$ denotes the sum over all $j\leq n$
with $(p,j) = 1$,\sn
$\bullet$ \ $\displaystyle\sum_{\nu}{}^{(j)}$ denotes the sum over all $\nu\geq 1$ with $jp^{\nu}\leq n$.

\parm
For large enough $\lambda\in\N$, the power
\begin{equation}          \label{^p}
\left(f_j(X_0) + \sum_{\nu}{}^{(j)} f_{jp^{\nu}}(X_0)^{p^{-\nu}}\right)^{p^{\lambda}}
\end{equation}
is a polynomial in $K[X_0]$. Since the approximation type of $x$ over $K$ is transcendental, we may choose
\[
\alpha_0\in v(x-K)      
\]
such that for all $c\in K$ with $v(x-c)\geq\alpha_0$ the value of $f(c)$ as well as the values of~(\ref{^p}) for
$X_0=c$ are fixed, for all $j\leq n$ with $(p,j) = 1$. For those $c$ we set
\begin{equation}
\beta_j \>:=\> v\left(f_j(c) + \sum_{\nu}{}^{(j)} f_{jp^{\nu}}(c)^{p^{-\nu}}\right)\>,
\end{equation}
which is an element of the $p$-divisible hull of $vK$. As the set $\{v(x-c)\mid c\in K\}$ has no greatest element,
we may choose $c$ with $v(x-c)\geq\alpha_0$  such that all values
\begin{equation}                            \label{betaimu}
\beta_j + j\cdot v(x-c)\,,\quad j\leq n\mbox{ with }(p,j)=1,
\end{equation}
are distinct, nonzero, and not equal to $vf(c)$. Having chosen $c$, we choose $d\in K$ such that $vd= v(x-c)$
and put
\[
z = \frac{x-c}{d}\>,
\]
hence $vz = 0$. In the above expressions we now set $X:=x$ and $X_0:=c$. Then $K'$ becomes a finite purely inseparable extension of $K$, and from (\ref{eqp}) we obtain a polynomial that may  be written as a polynomial
in $z$ as follows:
\begin{equation}
\tilde{g}(z) = f(c) + \sum_j {} ' \,b_j z^j
\end{equation}
with coefficients
\[
b_j \>=\> d^j\cdot\left(f_j(c) + \sum_{\nu}{}^{(j)} f_{jp^{\nu}}(c)^{p^{-\nu}}\right)\>,
\]
all of which have nonzero value. We note that $\tilde{g}(z)$ and $f(x)$ are equivalent modulo $\wp(K'[x])$.
\pars
If $vb_i > 0$ for some $i$, then $b_i z^i \in \wp(K'(z)^h)=\wp(K'(x)^h)$.
Consequently, $\tilde{g}(z)$ and thus also $f(x)$ are equivalent modulo $\wp(K'(x)^h)$ to a polynomial
\[
g(z) \>=\> a_n z^n + \ldots  + a_0 \in K'[x]
\]
where
\[
a_i \>=\> \left\{ \begin{array}{ll}
                  b_i & \mbox{if $vb_i < 0$} \\
                  0   & \mbox{otherwise}
                \end{array}
\right\}\qquad \mbox{for $1\leq i\leq n$.}
\]
In both polynomials $\tilde{g}$ and $g$, the coefficients $a_i$ and $b_i$
are equal to zero whenever $p$ divides $i$. On the other hand, the values of the nonzero coefficients $a_i$ for
$i>0$ are just the values given in (\ref{betaimu}), and by our construction, all of these values are distinct,
and different from $va_0\,$.

\parm
It remains to prove the last assertion for separably tame $(K,v)$. By \cite[Corollary~3.12]{[K7]},
such $(K,v)$ lies dense in its perfect hull $K^{1/p^{\infty}}$ and thus also in $(K',v)$. We choose
$a'_0,\ldots,a'_n \in K$ such that
the values $v(a_i-a'_i)$ are sufficiently large, with $a'_i=0$ if $a_i=0$, such that $va'_i<0$ if $va_i<0$ and
$v(g(z)-(a'_n z^n + \ldots  + a'_0))>0$. Then we may replace $g(z)$ by $a'_n z^n + \ldots  + a'_0\,$, hence
w.l.o.g.\ we may assume that $g$ has coefficients in $K$.
\end{proof}

\pars
Now we turn to the mixed characteristic case.
\begin{lemma}                    \label{,nff}
Assume that $\chara K=0$ and that $K$ is closed under $p$-th roots. Then for every $f(x)\in K[x]$ with $vf(x)>0$
there exists a polynomial
\begin{equation}                     \label{,Ke2}
g(z)\in (1+f(x))\cdot (K(x)^h)^p
\end{equation}
satisfying:
\begin{equation}                     \label{,nfh}
\left\{ \begin{minipage}{0.85\textwidth}\noindent
$g(z)=a_n z^n + \ldots  + a_1 z + a_0 \in {\cal M}_K [z] \mbox{ where}$\n
$z = (x - c)/d\,,\mbox{ with } vz = 0\,,\; c\in K \mbox{ and } 0\not= d \in K$,\n
$va_i>\frac{p}{p-1}vp \>\Rightarrow\> a_i=0$, and\n
there is $i_0\in\{1,\ldots,n\}$ with $p\nmid i_0$ such that $a_{i_0}$ is the unique coefficient of least value
among $a_1,\ldots,a_n\,$.
\end{minipage}
\right.
\end{equation}
Moreover, we may assume that $p\nmid j$ whenever $va_j\leq vp$, and we may assume it to hold for all nonzero
$a_j$ if $va_{i_0}>vp$. In the latter case, we may even assume that all nonzero coefficients have distinct value.
\end{lemma}
\begin{proof}
We will alter the polynomial $f(x)$ in several subsequent steps.

We again use the Taylor expansion (\ref{TE}). As before,
since the approximation type of $x$ over $K$ is transcendental, we may choose $\alpha_0\in v(x-K)$
such that for all $c\in K$ with $v(x-c)\geq\alpha_0$ the values of $f_i(c)$ are fixed, for every $i$. As the set
$v(x-K)$ has no greatest element, we may choose
$\alpha_0$ so large that for all $c$ with $v(x-c)\geq\alpha_0$ the values of all monomials $f_i(c)(x - c)^i$
will be distinct and there is $i_0\geq 1$ such that
\[
 vf_{i_0}(c)(x - c)^{i_0}\><\>vf_i(c)(x - c)^i\quad \mbox{ for all $i\geq 1$, $i\ne i_0$}
\]
(cf.\ \cite[Lemma~5.1 and Lemma~5.2] {[K--V]}). Hence we obtain from the Taylor expansion:
\[
v(f(x)-f(c))\>=\>v\sum_{i=1}^n f_i(c)(x - c)^i\>=\> vf_{i_0}(c)(x - c)^{i_0}\>.
\]
By our choice of $c$, the values $vf(c)$ and $vf_{i_0}(c)$ are fixed, and also $vf(x)$ is a constant. But as $v(x-K)$ has no maximal element, the value of the right hand side is not fixed. This can only be if $vf(x)=vf(c)$, and we obtain that for large enough $\alpha_0\,$,
\[
0\><\>vf(x)\>=\>vf(c)\><\>vf_i(c)(x - c)^i \quad \mbox{ for $1\leq i\leq n$.}
\]
That is, $f_i(c)(x - c)^i\in {\cal M}_{K(x)}$ for all $i\geq 0$.

\pars
For large enough $\alpha_0$ we can further assume: if $j=p^t$ and $i=p^t r$ with $t\geq 0$ and $r>1$, $(r,p)=1$,
then
\begin{equation}                      \label{,ij}
vf_j(c)(x - c)^j\><\>vf_i(c)(x - c)^i
\end{equation}
(unless both values are equal to $\infty$). This is shown in \cite[Lemma 7]{[Ka]}
(see also \cite[Proposition~7.4]{[K--V]}).

\pars
Fix any $c_1\in K$ with $v(x-c_1)\geq\alpha_0$. As $(K(x)|K,v)$ is immediate, there is some
$a\in K$ with $v((x-c_1)/a)=0$. We set $y=(x-c_1)/a$ and $d_i=f_i(c_1) a^i$, so that
\[
f(x)\>=\> \sum_{i=0}^n d_i y^i\>.
\]
Note that $K[x]=K[y]$ and that $vy=0$, whence $vd_iy^i=vd_i>0$.

For every $i$ divisible by~$p$, we choose $d_i^{1/p}$ to be any $p$-th
root of $d_i$ in $K$ (see our discussion following Lemma~\ref{1+y}); this is possible
since $K$ is closed under $p$-th roots by assumption. Then we have:
\begin{equation}              \label{,i/p}
d_i^{1/p} y^{i/p} \in {\cal M}_K [y] \subseteq {\cal M}_{K(y)}\;.
\end{equation}
We choose a polynomial $s(y)\in K[y]$ such that
\[
s(y) \>\equiv\> (1 + \sum_{p|i} d_i^{1/p} y^{i/p})^{-1}
\;\;\mbox{\rm mod } p {\cal M}_K [y]\;;
\]
this can be done using the geometrical series of the right hand side together with our assumption that the rank of
\mbox{$(K,v)$} is 1. Note that $vs(y)=0$ and the constant term of $s(y)$ as a polynomial in $y$ is a $1$-unit,
so also the constant term of $s(y)^p$ is a $1$-unit. We have that
\[
s(y)^p \>\equiv\> (1 + \sum_{p|i} d_i^{1/p} y^{i/p})^{-p} \;\;\mbox{\rm mod } p^2{\cal M}_K [y]
\]
and
\begin{eqnarray*}
s(y)^p(1+f(x)) & = & s(y)^p\left(1+\sum_{p|i} d_iy^i\,+\,\sum_{p\nmid i} d_iy^i\right)\\
 & = & s(y)^p\left[(1+\sum_{p|i}d_i^{1/p}y^{i/p})^p\,+\,\tilde{f}(y)\,+\,
\sum_{p\nmid i} d_iy^i\right]
\end{eqnarray*}
with
\[
\tilde{f}(y)\in p{\cal M}_K [y]\>.
\]
Modulo $p^2{\cal M}_K [y]$, $s(y)^p(1+f(x))$ is hence equivalent to
\[
1\,+\,s(y)^p\tilde{f}(y) \,+\, s(y)^p\sum_{p\nmid i} d_iy^i \>.
\]
In (\ref{,Ke2}) we can replace $1+f$ by $s(y)^p(1+f)$ without changing the right hand side. Since
$vp^2=2vp\geq\frac{p}{p-1}vp$ for all primes $p$, part a) of Lemma~\ref{1+y} shows that we can further
replace $s(y)^p(1+f)$ by $1+\tilde{g}(y)$ with
\begin{equation}              \label{case1}
\tilde{g}(y) \>:=\> s(y)^p\tilde{f}(y)+s(y)^p\sum_{p\nmid i} d_iy^i \>.
\end{equation}

\pars
Now we distinguish two cases. Let us assume first that there is a monomial in the polynomial
that has a value $\leq vp$. Since $\tilde{f}(y)\in p{\cal M}_K [y]$ and $vs(y)^p=pvs(y)=0$, we find that also
$s(y)^p\tilde{f}(y)\in p{\cal M}_K [y]$. So the momomials of value $\leq vp$ must come from the sum
$\sum_{p\nmid i} d_iy^i$, which consequently is nonempty. By (\ref{,ij}), the monomial $d_1y$ is the unique
one of minimal value in this sum. Hence $vd_1y\leq vp$.

Let $d'_1$ be the coefficient of $y$ in $\tilde{g}(y)$. Since $s(y)^p\tilde{f}(y)\in p{\cal M}_K[y]$ and the 
constant term of $s(y)^p$ as a polynomial in $y$ is a $1$-unit, it follows that $v(d'_1-d_1)>vd_1$,
whence $vd'_1=vd_1$, and that $d'_1y$ is the unique summand of minimal value in $\tilde{g}(y)$. Every summand of
value greater than $\frac{p}{p-1}vp$ can be deleted by part a) of Lemma~\ref{1+y}. Setting $z=y$, $c=c_1$ and
$d=a$, we arrive at a polynomial $g(z)=\tilde{g}(y)$ and elements $c$ and $d$ which satisfy the assertion
of our lemma in the first case.

\parb
Now we consider the second case: all monomials in $\tilde{g}(y)$ have value $>vp$. Since $vy=0$, this implies that
$\tilde{g}(Y)\in p{\cal M}_K [Y]$.

We will work with polynomials of the following form:
\begin{equation}           \label{+TE}
h^*(Y,Z) \>=\> \sum_{i=0}^n h_i(Z)(Y - Z)^i \>\in\> p{\cal M}_K [Y,Z] \>.
\end{equation}
Note that here we use ``$i$'' only in the sense of an ordinary index, and not to denote a Hasse-Schmidt derivative.

For $h^*(Y,Z)$ as in (\ref{+TE}) and $\alpha\in v(y-K)$ we will call $(h^*(Y,Z),\alpha)$ an
\bfind{admissible pair} if for all $c\in K$ with $v(y-c)\geq\alpha$,
\sn
a) \ the values $vh_i(c)$ are fixed for all $i$,
\n
b) \ the values of the nonzero summands $h_i(c)(y-c)^i$ are distinct,
\n
c) \ $1+ h^*(y,c) \>\in\> (1+f)\cdot (K(x)^h)^p$.
\sn
Then for large enough $\alpha$, $(\sum_{i=0}^n \tilde{g}_i(Z)(Y - Z)^i,\alpha)$, where in this case each
$\tilde{g}_i$ denotes the $i$-th Hasse-Schmidt derivative of $\tilde{g}$, is an admissible pair.

\pars
Let us fix an admissible pair $(h^*(Y,Z),\alpha)$ for which the number of nonzero monomials in $h^*(Y,Z)$ is the
smallest possible. Take any $c\in K$ such that $v(y-c)\geq\alpha$. If $m\geq 1$, $(p,i)=1$, and the summand
$h_{ip^m}(c)(y-c)^{ip^m}$ is nonzero, then we use part c) of Lemma~\ref{1+y} $m$ many times to replace this
summand by
\[
\Delta^m (h_{ip^m}(c)(y-c)^{ip^m}) \>=\> \Delta^m(1)\cdot h_{ip^m}(c)^{1/p^m}\,(y-c)^i\>.
\]
In this way we turn $h^*(y,c)$ into a polynomial
\begin{equation}
h_0(c)+\sum_{1\leq i\leq n\,,\,p\nmid i} r_i(c)(y-c)^i
\end{equation}
where
\[
r_i(c)\>=\> \sum_{m\geq 0\,,\, ip^m\leq n} \Delta^m(1)\cdot h_{ip^m}(c)^{1/p^m} \>\in\> K\>.
\]
Note that $r_i(Y)$ is not necessarily a polynomial. Nevertheless, we wish to show that for each $i$ the value
$vr_i(c)$ is fixed for all $c\in K$ with $v(y-c)\geq\alpha$. We observe that the values
\[
v\,h_{ip^m}(c)^{1/p^m} \>=\> \frac{1}{p^m} v h_{ip^m}(c)
\]
are fixed for all those $c$. Therefore it suffices to show that the values
\[
v \Delta^m (h_{ip^m}(c)) \>=\> \delta^m (v h_{ip^m}(c))\,,\quad m\geq 0\,,\, ip^m\leq n
\]
are distinct because then $vr_i(c)$ is equal to the minimum of these values and is consequently fixed for our
choices of $c$.

Suppose that there are $\ell>m\geq 0$ with $ip^\ell\leq n$ such that $\delta^\ell (v h_{ip^\ell}(c))=
\delta^m(v h_{ip^m}(c))$. This implies that
\[
\delta^{m-\ell} (v h_{ip^m}(c)(y-c)^{ip^m}) \>=\> v h_{ip^\ell}(c)(y-c)^{ip^\ell}\> > vp\>.
\]
Thus by (\ref{Delta^-1}), $\Delta^{-1}$ can be applied $\ell -m$ many times in order to replace the summand
$h_{ip^m}(c)(y-c)^{ip^m}$ by $\Delta^{m-\ell}(1)\cdot h_{ip^m}(c)^{p^{\ell-m}}(y-c)^{ip^\ell}$.
Similarly, we replace the monomial $h_{ip^m}(Z)(Y-Z)^{ip^m}$ by $\Delta^{m-\ell}(1)\cdot
h_{ip^m}(Z)^{p^{\ell-m}}(Y-Z)^{ip^\ell}$ in $h^*(Y,Z)$; this is possible since $h_{ip^m}(Z)^{p^{\ell-m}}$ is again
a polynomial in $Y$. This procedure turns the coefficient of $(Y-Z)^{ip^m}$ into $0$ while the coefficient of
$(Y-Z)^{ip^\ell}$ becomes $\Delta^{m-\ell}(1)\cdot h_{ip^m}(Z)^{p^{\ell-m}}+h_{ip^\ell}(Z)$. So the new polynomial
$\tilde{h}^*$, say, has less monomials than $h^*$.

We choose $\alpha'\in v(y-K)$ so large that also the value $v(\Delta^m(1)h_{ip^\ell}(c)^{p^{\ell-m}}+h_{ip^m}(c))$
of the new coefficient is fixed and the values of all summands are distinct whenever $v(y-c)\geq\alpha'$. Then
$(\tilde{h}^*(Y,Z),\alpha')$ is an admissible pair, contradicting the minimality of $(h^*(Y,Z),\alpha)$. This
completes the proof of the fact that the values of the coefficients $r_i(c)$ are fixed for all $c\in K$ with
$v(y-c)\geq\alpha$.

Again, as $v(y-K)$ has no largest element, we can choose some $c_2\in K$ with
$v(y-c_2)$ so large that the values $vr_i(c_2)(y-c_2)^i$ are distinct.
As in the first case, every summand of value greater than $\frac{p}{p-1}vp$ can be deleted by part a) of
Lemma~\ref{1+y}. As $(K(x)|K,v)$ is immediate, there is some $b\in K$ with $v((y-c_2)/b)=0$. We set $z=(y-c_2)/b$
and $a_i=r_i(c_2)b^i\in {\mathcal M}_K\,$. Then $vz=0$. Further, $z=(x-c)/d$ for $c=c_1+ac_2$ and $d=ab$.
Hence also in the second case, the assertions of our lemma are satisfied.
\end{proof}

\begin{remark}
In the first case, the proof yields $a_1$ to be the coefficient of minimal value. This can also be achieved in
the second case by an application of \cite[Lemma 7.6]{[K--V]}. But then apparently we may not achieve that only
those coefficients $a_i$ are nonzero for which $p\nmid i$.
\end{remark}

\begin{remark}                     \label{remershov}
The original proof given in \cite{[K1]} for the second case contained a mistake, which was noticed by Yuri Ershov.
In the paper \cite{[Er]} Ershov suggests an improved approach and fills the gap. We have taken over from this paper the
very helpful instrument of the function $\delta$, as well as the idea to consider ``admissible pairs''. However, we have
chosen an enhanced definition of ``admissible pair'' and have replaced Ershov's concept of ``normal pair'' by working with admissible pairs with a minimal number of monomials, which simplifies the proof.
\end{remark}

%
%
\subsection{Structure of $E$ under suitable assumptions on $K$}               
In order to apply the normal forms that we have found to determine the structure of the Galois extensions in
question, we need some preparations. First, we will need to show that the condition on the approximation type of
the element $x$ is satisfied in the situations we are going to consider.

\begin{lemma}                      \label{transat}
Take an immediate transcendental extension $(K(x)|K,v)$ and assume that $K$ is separable-algebraically maximal.
Then the approximation type of $x$ over $K$ is transcendental.
\end{lemma}
\begin{proof}
It is shown in the proof of \cite[Proposition~3.10]{[KPal]} that under the assumptions of our lemma, the following
holds: every pseudo Cauchy sequence in $(K,v)$ with limit $x$ and without a limit in $K$ is of transcendental
type; for these notions, see \cite{[Ka]}. This implies (and in fact is equivalent to) that the approximation type
of $x$ over $K$ is transcendental. Indeed, if it were not, then one could construct a pseudo Cauchy sequence in
$(K,v)$ with limit $x$ and without a limit in $K$ of transcendental type by (possibly transfinite) induction,
since there will be some $f\in K[X]$ such that for every $c\in K$ there is $c'\in K$ with $v(x-c')>v(x-c)$ and
$vf(c')\ne vf(c)$.
\end{proof}

We will also need the following result which is a consequence of \cite[Theorem 9.1 in conjunction with Corollary
7.7]{[K--V]}. A direct proof can also be found in \cite{[Er]}.
\begin{lemma}                          \label{zh=fzh}
Assume that the extension $(K(z)|K,v)$ is immediate with $vz=0$, and that the approximation type of $z$ over $K$
is transcendental. Take a polynomial $f(X)= a_nX^n+\ldots+a_0\in K[X]$ for which there is an index
$i_0\in\{1,\ldots,n\}$ with $p\nmid i_0$ such that $a_{i_0}$ is the unique coefficient of least value
among $a_1,\ldots,a_n$. Then
\[
K(z)^h \>=\> K(f(z)^h)\>.
\]
\end{lemma}

\pars
Now we can prove:
\begin{proposition}                     \label{xh=fxheq}
Take a separably tame valued field $(K,v)$ of characteristic $p>0$ and rank 1, an immediate
transcendental extension $(K(x)|K,v)$, and a Galois extension $E$ of $K(x)^h$ of degree $p$.
Then there exists $\vartheta\in E$ such that $E=K(\vartheta)^h$.
\end{proposition}
\begin{proof}
We can assume that (\ref{ASextf}) holds. Since a separably tame field is separable-algebraically maximal by
\cite[Theorem~3.10]{[K7]}, Lemma~\ref{transat} shows that the approximation type of $x$ over $K$ is
transcendental. Because of Lemma~\ref{nff} we can assume that
that $\vartheta^p-\vartheta=g(z)$, where $g(z)$ is as in (\ref{nfh}). We note that $g(z)\in K[\vartheta]$,
so $K(g(z))^h\subseteq K(\vartheta)^h$. From Lemma~\ref{zh=fzh} we infer that $K(z)^h=K(g(z))^h$, whence
\[
K(x)^h \>=\> K(z)^h \>=\> K(g(z))^h \>\subseteq\> K(\vartheta)^h \>.
\]
Therefore,
\[
E \>=\> K(x)^h(\vartheta)  \>=\> K(x,\vartheta)^h  \>=\> K(\vartheta)^h\>,
\]
as desired.
\end{proof}

\parm
In the mixed characteristic case, we have the following:
\begin{proposition}                     \label{xh=fxhmix}
Take an algebraically closed valued field $(K,v)$ of characteristic 0 and rank 1, an immediate transcendental
extension $(K(x)|K,v)$, and a Galois extension $E$ of $K(x)^h$ of degree $p=\chara Kv>0$. Then there is some
$\eta\in E$ such that $E=K(\eta)^h$.
\end{proposition}
\begin{proof}
Since $K$ is algebraically closed, it contains the $p$-th roots of unity and is also closed under $p$-th roots. So
we can assume that (\ref{f(x)}) holds. Since an algebraically closed valued field is obviously
separable-algebraically maximal, Lemma~\ref{transat} again shows that the approximation type of $x$ over $K$ is
transcendental. Because of Lemma~\ref{,nff} we can assume that
that $\eta^p=1+g(z)$, where $g(z)$ is as in (\ref{,nfh}). Again we have that $g(z)\in K[\eta]$,
so $K(g(z))^h\subseteq K(\eta)^h$. From Lemma~\ref{zh=fzh} we infer that $K(z)^h=K(g(z))^h$ and conclude that
$E = K(\eta)^h$ as in the foregoing proof.
\end{proof}

%
%
\section{Proof of Theorem~\ref{stt3'} and Proposition~\ref{N}}                        
In this section, we will build up the proof of Theorem~\ref{stt3'} step by step.
We will at first concentrate on the case of transcendence degree 1. The proof for the case
of higher transcendence degree and the proof of Proposition~\ref{N} will then be given at the end of this section.

%
%
\subsection{Separable-algebraically closed base fields of rank 1}                   
In this subsection, we will prove that every separable immediate
function field of transcendence degree 1 over a separable-algebraically closed
base field of rank 1 is henselian rational. The following result is instrumental in
the reduction to Galois extensions of degree $p$:
\begin{lemma}                               \label{,mcc}
Let $(L,v)$ be a henselian field of characteristic $p > 0$ with divisible value group and algebraically closed
residue field. Then every
nontrivial finite separable extension of $L$ is a tower of Galois extensions of degree $p$.
\end{lemma}
\begin{proof}
From our conditions on value group and residue field, it follows that the separable-algebraic closure of $L$ is
an immediate extension of $(L,v)$. Hence by the Lemma of Ostrowski the degree of every finite subextension is
a power of $p$. This shows that the separable-algebraic closure of $L$ is a $p$-extension of $L$.
It follows from the general theory of $p$-groups (cf.\ \cite{[H]}, Chapter III, \S 7, Satz 7.2 and the following
remark) via Galois correspondence that every finite subextension of a $p$-extension
is a tower of Galois extensions of degree $p$. This implies the assertion of our lemma.
\end{proof}

With the help of this lemma, the results we proved in Section~\ref{sectGalext} allow us to take the first step
towards our main theorem:
\begin{proposition}         \label{,stt1}
Every immediate separable function field $(F|K,v)$ of transcendence degree 1 over a separable-algebraically
closed field $(K,v)$ of rank 1 is henselian rational.
\end{proposition}
\begin{proof}
The valuation $v$ is nontrivial on $K$ since otherwise $F=K$ because $(F|K,v)$ is immediate. As
$K$ is separable-algebraically closed, it follows that $vK$ is divisible and $Kv$ is algebraically closed
(cf.\ \cite[Lemma 2.16]{KTrans}).

We choose a separating element $x$ of $F|K$. Since the subextension $(K(x)|K,v)$ of $(F|K,v)$ is immediate, we have
that $vK(x)=vK$ is divisible and $K(x)v=Kv$ is algebraically closed.
If $F\subseteq K(x)^h$ does not already hold, then by Lemma~\ref{,mcc}
the nontrivial finite separable extension $F.K(x)^h|K(x)^h$  is a tower of Galois extensions of degree $p$. By
induction on the number of Galois extensions in the tower, using Proposition~\ref{xh=fxheq} or
Proposition~\ref{xh=fxhmix} respectively, we find $y\in F^h$ such that $F.K(x)^h=K(y)^h$ and therefore,
$F\subseteq K(y)^h$. (Proposition~\ref{xh=fxheq} can be applied because every
separable-algebraically closed valued field is separably tame, and Proposition~\ref{xh=fxhmix} can be applied
because every separable-algebraically closed field of characteristic 0 is algebraically closed.)

Suppose that $y$ lies in the completion $K^c$ of $(K,v)$. Since $K$ is separable-algebraically closed, $(K,v)$ is
henselian, and so is its completion. Thus $K(y)$ and hence also $F$ can be assumed embedded in $K^c$. Then it
follows from Theorem~\ref{Xicr} that $F\subset K(x)^h$.

Suppose now that $y\notin K^c$. Then \cite[Theorem 11.1]{[K--V]} shows that $y$ can already be chosen in $F$,
which proves that $(F|K,v)$ is henselian rational.
\end{proof}

%
%
\subsection{Separably tame base fields of rank 1}                               \label{sectred2}
We will now generalize Proposition~\ref{,stt1} to the case of separably tame base fields of rank 1.
Theorem~\ref{,pdp} is the same as Theorem 14.5 of \cite{[K--V]}, except that the latter assumes that $(K,v)$ is
algebraically maximal. However, all that is needed for the proof of that theorem is that if $x$ is transcendental
over $K$, then the approximation type of $x$
over $(K,v)$ is transcendental. If $(K,v)$ is separable-algebraically maximal,
then this follows from Lemma~\ref{transat}. Thus Theorem~\ref{,pdp} is proven.
\begin{proposition}                     \label{,hr1}
Every immediate separable function field $(F|K,v)$ of transcendence degree 1 over a separably tame
field $(K,v)$ of rank 1 is henselian rational.
\end{proposition}
\begin{proof}
If $F$ lies in the completion of $K$, our assertion follows from Theorem~\ref{Xicr}; so let us assume now that $F$
is not contained in the completion of $K$. We know that $F.K\sep|K\sep$ is henselian
rational by virtue of Proposition~\ref{,stt1}. As $(K,v)$ is a separably tame field, the
extension $(K\sep|K,v)$ is tame by definition. Hence our assertion follows from Theorem~\ref{,pdp}.
\end{proof}

%
%
\subsection{Separably tame base fields of finite rank}                              
The next step towards the desired structure theorem is the generalization of Proposition~\ref{,hr1}
to the case of finite rank. We let $P$ be the place that is associated with the valuation $v$ on $F$.
We need some preparations.
\begin{lemma}                               \label{xysep}
Let $F|K$ be a separable function field of transcendence degree 1 and $Q$
a nontrivial place on $F$. Then for every $x\in F^\times$ there exists a
separating element $y$ of $F|K$ which satisfies $v_Q(y)=v_Q(x)$ and if $v_Q(x)=0$, also $yQ=xQ$.
\end{lemma}
\begin{proof}
Let us choose any separating element $z$ of $F|K$ and an element
$a\in F^\times$ which satisfies $v_Q(a)>v_Q(x)$ and $v_Q(az)>v_Q(x)$.
Note that $F|K(x,z,a)$ and $K(x,z,a)|K$ are separable extensions. We have:
\[
K(x,z,a) \>=\> K(x,x+az,x+a)
\]
with
\[
v_Qx\>=\>v_Q(x+az)\>=\>v_Q(x+a) \;\mbox{ and }\; xQ\>=\>(x+az)Q\>=\>(x+a)Q\;.
\]
On the other hand, at least one of the elements $x$, $x+az$, $x+a$ must be a separating element for
the separable extension $K(x,x+az,x+a)|K$ (cf.\ \cite[VIII, \$4, Proposition 4.8]{[L]});
we take $y$ to be such an element. Then $y$ is also a separating element
for $F|K$, and it satisfies our assertion.
\end{proof}

The following lemma will provide some useful information for the case that $P$ is the composition of two
nontrivial places $Q_1$ and $Q_2$, where $Q_2$ is a place on the residue field $FQ_1\,$; we write $P=Q_1Q_2\,$.
The henselization of $(F,Q_1)$ will be denoted by $(F^{h(Q_1)},Q_1)$; similarly, $((FQ_1)^{h(Q_2)},Q_2)$
indicates the henselization of $(FQ_1,Q_2)$. Note that the extension of $Q_1$ from $F^{h(Q_1)}$ to $F^{h(P)}$
is unique; we denote it again by $Q_1\,$.
\begin{lemma}                  \label{dec}
Take any valued field $(L,P)$ with $P=Q_1Q_2\,$.
\sn
a) Take any field extension $L'|L$ and extend $P,Q_1,Q_2$ to $L'$ such that $P=Q_1Q_2$ also holds on $L'$.
Then $(L'|L,P)$ is immediate if and only if $v_{Q_1} L' =v_{Q_1} L$ and the extension
$(L'Q_1|LQ_1,Q_2)$ is immediate.
\sn
b) The valued field $(L,P)$ is henselian if and only if $(L,Q_1)$ and $(LQ_1,Q_2)$ are.
\sn
c) The extension $(L^{h(P)}|L^{h(Q_1)},Q_1)$ is tame, and $(L^{h(P)})Q_1=(LQ_1)^{h(Q_2)}$.
\end{lemma}
\begin{proof}
The straightforward proof of part a) is left to the reader. Part b) is \cite[Theorem~32.15]{[W]} (where it is
stated using valuations instead of places). For the proof of part b) one uses the fact that
a valued field is henselian if and only if the extension of its valuation to the algebraic closure is unique.

We prove part c).
There exists a tame algebraic extension $(E,Q_1)$
of $(L^{h(Q_1)},Q_1)$ such that $EQ_1=(LQ_1)^{h(Q_2)}$; this is found as follows.
The \bfind{absolute inertia field} of $(L^{h(Q_1)},Q_1)$, which we denote by $(L^{i(Q_1)},Q_1)$, is the inertia
field of the normal extension $(L\sep|L^{h(Q_1)},Q_1)$. It is a subfield of the absolute ramification field
of $(L^{h(Q_1)},Q_1)$ and is therefore a tame extension of $(L^{h(Q_1)},Q_1)$. The henselization $(LQ_1)^{h(Q_2)}$
is a separable-algebraic extension of $L^{h(Q_1)}=LQ_1\,$. Hence by \cite[part (2) of Theorem 5.2.7]{[Eng--P]}
there is a subextension $(E|L^{h(Q_1)},Q_1)$ of $(L^{i(Q_1)}|L^{h(Q_1)},Q_1)$ such that $EQ_1=(LQ_1)^{h(Q_2)}$.
As a subextension of a tame extension, this extension is tame as well. Once we have shown that
$E=L^{h(P)}$, part c) of our lemma is proven.

As an extension of the henselian field $(L^{h(Q_1)},Q_1)$, also $(E,Q_1)$ is henselian. Further, $(EQ_1,Q_2)=
((LQ_1)^{h(Q_2)},Q_2)$ is henselian. Hence by part b), $(E,P)$ is henselian and must therefore contain
$L^{h(P)}$, so $(L^{h(P)}|L^{h(Q_1)},Q_1)$ is a subextension of the absolute inertia field. We observe that
$(L^{h(P)}Q_1,Q_2)$ is henselian by part b), so it must contain the henselization $(LQ_1)^{h(Q_2)}=EQ_1\,$.
Again from \cite[part (2) of Theorem 5.2.7]{[Eng--P]} we obtain that $E\subseteq L^{h(P)}$, so equality holds.
\end{proof}

\begin{proposition}               \label{stt2}
Every immediate separable function field $(F|K,v)$ of transcendence degree 1 over a separably tame
field $(K,v)$ of finite rank is henselian rational.
\end{proposition}
\begin{proof}
Let $P$ denote the place associated with $v$.
Since $F$ has finite rank and $F|K$ is an immediate extension of transcendence
degree 1, there exist places $P_1$, $P_2$, $P_3$ where $P_1$ and $P_3$
may be trivial and $P_2$ has rank 1, such that $P = P_1P_2P_3$ and
\begin{eqnarray*}
&& \trdeg (FP_1|KP_1) = 1\;,\\
&& \trdeg (FP_1P_2|KP_1P_2) = 0\;.
\end{eqnarray*}
By \cite[Lemma~3.14]{[K7]}, the hypothesis that $(K,P)$ is a separably tame field yields
\sn
1) $(K,P_1)$ and $(KP_1,P_2)$ are separably tame fields,
\n
2) if $P_1$ is nontrivial, then $(KP_1,P_2)$ is a tame field,
\n
3) since $P_2$ is nontrivial, the same holds for $P_1P_2$ and $(KP_1P_2,P_3)$ is a tame field.
\pars
Since $(F|K,P)$ is immediate, it follows from Lemma~\ref{dec} that $v_{P_1} F =v_{P_1}K$,
$v_{P_2}(FP_1)=v_{P_2}(KP_1)$, and that the algebraic extension
$(FP_1P_2|KP_1P_2,P_3)$ is immediate. Since the tame field $(KP_1P_2,P_3)$ is defectless, the latter
extension must be trivial. This yields that also
\begin{equation}                      \label{FP1KP1}
(FP_1|KP_1,P_2)
\end{equation}
is an immediate extension. Since $\trdeg F|K=1=\trdeg FP_1|KP_2\,$, \cite[Corollary~2.3]{[K7]} can be applied to
$(F|K,P_1)$ to deduce that $FP_1|KP_1$ is finitely generated. We have shown that (\ref{FP1KP1}) is an immediate
function field of transcendence degree 1 and rank 1.

Now we distinguish two cases. If $P_1$ is nontrivial, then $(KP_1,P_2)$ is a tame field and hence perfect. It
follows that $FP_1|KP_1$ is separable. If $P_1$ is trivial, then since $F|K$ is separable by assumption,
also $FP_1|KP_1$ is separable. In both cases, we can now apply Proposition~\ref{,hr1} to obtain that
(\ref{FP1KP1}) is henselian rational. So we may write
\[
FP_1^{h(P_2)} \>=\> KP_1(xP_1)^{h(P_2)}
\]
for a suitable
$x\in F$ which is consequently transcendental over $K$ with $xP_1$ transcendental over $KP_1\,$. Applying
Lemma~\ref{xysep} with $Q=P_1$ we see that $x$ can be chosen to be a separating element for $F|K$, so that
$F|K(x)$ is a finite separable extension. Our goal is to show that $F^{h(P)}=K(x)^{h(P)}$.

\pars
We know that $F^{h(P_1P_2)}P_1P_2=FP_1P_2=KP_1P_2$ and that $K(x)^{h(P_1P_2)}P_1P_2=$ $K(x)P_1P_2=KP_1P_2$ (the
latter holds since $KP_1P_2\subseteq K(x)P_1P_2\subseteq FP_1P_2 = KP_1P_2$). Since $(KP_1P_2,P_3)$ is a tame field,
it is henselian. By part b) of Lemma~\ref{dec} it follows that $(F^{h(P_1P_2)},P)$ is henselian, so it must contain
$F^{h(P)}$. It also follows that $(F^{h(P)}, P_1P_2)$ is henselian, so it must contain $F^{h(P_1P_2)}$. Hence the
two fields are equal. In the same way one shows that $K(x)^{h(P)}=K(x)^{h(P_1P_2)}$. Using this together with
part c) of Lemma~\ref{dec} we obtain that $F^{h(P)}P_1=F^{h(P_1P_2)}P_1=FP_1^{h(P_2)}$ and
$K(x)^{h(P)}P_1=K(x)^{h(P_1P_2)}P_1=K(x)P_1^{h(P_2)}\supseteq KP_1(xP_1)^{h(P_2)}$. Altogether,
\begin{equation}                   \label{samerf}
K(x)^{h(P)}P_1 \>\subseteq\> F^{h(P)}P_1 \>=\> FP_1^{h(P_2)} \>=\> KP_1(xP_1)^{h(P_2)}
\>\subseteq\> K(x)^{h(P)}P_1\>,
\end{equation}
hence equality holds everywhere. If $P_1$ is trivial, then this implies that $F^{h(P)}=K(x)^{h(P)}$ and we are
done. We wish to show the same in case that $P_1$ is nontrivial.

\pars
Since the extensions $(F|K,P)$ and $(F^{h(P)}|F,P)$ are immediate, so is $(F^{h(P)}|K,P)$.
By part a) of Lemma~\ref{dec} this yields that $v_{P_1}F^{h(P)}= v_{P_1}K$. Thus $v_{P_1}K\subseteq v_{P_1}
K(x)^{h(P)}\subseteq v_{P_1}F^{h(P)}= v_{P_1}K$, so equality holds everywhere. Together with the equality
following from equation
(\ref{samerf}), this proves that the extension
\[
(F^{h(P)}|K(x)^{h(P)},P_1)
\]
is immediate. We wish to show that this extension is trivial.

We have already noted that $(K,P_1)$ is a separably tame and hence separably defectless field.
Since $xP_1$ is transcendental over $KP_1\,$, we can apply \cite[Theorem~1]{[K8]} to find that
$(K(x),P_1)$ is a separably defectless field. By \cite[Theorem~(18.2)]{[End]},
also $(K(x)^{h(P_1)},P_1)$ is a separably defectless field. We infer from part c) of Lemma~\ref{dec}
that $(K(x)^{h(P)}|K(x)^{h(P_1)},P_1)$ is a tame extension. Hence by \cite[Proposition~2.12]{[K7]}
also $(K(x)^{h(P)},P_1)$ is a separably defectless field.

Since the extension $F|K(x)$ is finite and separable, Lemma~\ref{FhFEh} shows that the same is true for the
extension $(F^{h(P)}|K(x)^{h(P)},P_1)$. As this extension is also immediate and $(K(x)^{h(P)},P_1)$ is a
henselian separably defectless field, it follows that the extension must be trivial, as desired.
\end{proof}

%
%
\subsection{Separably tame base fields of arbitrary rank}                              
Now we are able to prove Theorem~\ref{stt3'} for the case of transcendence degree 1.
\begin{proposition}                          \label{sttarb}
Every immediate separable function field $(F|K,v)$ of transcendence degree 1 over a separably tame
field $(K,v)$ of arbitrary rank is henselian rational.
\end{proposition}
\begin{proof}
According to \cite[Corollaries~3.8 and~3.16]{[K7]} there exists a separably tame
subfield $K_0$ of $K$ of finite rank and a function field $F_0$ of transcendence degree 1 over $K_0$ with
$K_0v=Kv$ and $vK/vK_0$ torsion free, such that $F=F_0.K$ and that $vK_0$ is cofinal
in $vF_0$; since $F|K$ is assumed to be separable, we may also
assume $F_0|K_0$ to be separable. If we are able to show that
\begin{equation}                                      \label{as}
F_0^h \>=\> K_0(x)^h
\end{equation}
for some $x\in F_0\subseteq F$, then it will follow that
\[
F^h \>=\> (F_0.K)^h \>=\> (F_0^h.K)^h \>=\> (K_0(x)^h.K)^h \>=\> (K_0(x).K)^h \>=\> K(x)^h\>,
\]
and our proposition will be proved.

\pars
If $(F_0|K_0,v)$ is immediate, then the existence of $x\in F_0$ satisfying (\ref{as}) follows
from Proposition~\ref{stt2}. 
Let us assume now that $(F_0|K_0,v)$ is not immediate.

We have:
\[
K_0v \>\subseteq\> F_0v \>\subseteq\> Fv \>=\> Kv \>=\> K_0v\>,
\]
so equality holds everywhere. In particular we have that $K_0v=F_0v$ and thus $vK_0\ne vF_0$ by our assumption.
Since $vK/vK_0$ is
torsion free and $vF_0\subseteq vF=vK$, also $vF_0/vK_0$ is torsion free. Therefore, $\trdeg F_0|K_0 =1$ is equal
to the rational rank of $vF_0/vK_0$ and we can employ \cite[Corollary~2.3]{[K7]} to obtain that the torsion free
group $vF_0/vK_0$ is finitely generated. It follows that
\[
vF_0 \>=\> vK_0 \oplus \Z vx \>=\> vK_0(x)
\]
for a suitable $x\in F_0\,$. By Lemma~\ref{xysep} applied with $v_Q=v$, we may choose $x$ to
be a separating element of $F_0|K_0\,$. As $K_0v\subseteq K_0(x)v\subseteq F_0v \subseteq K_0v$, equality holds
everywhere. Therefore, $(F_0|K_0(x),v)$ is a finite immediate and separable extension. By Lemma~\ref{FhFEh},
the same holds for $F_0^h|K_0(x)^h$. According to \cite[Theorem 1]{[K8]}, $(K_0(x),v)$ is a separably
defectless field, and the same is true for $(K_0(x)^h,v)$ by \cite[Theorem~(18.2)]{[End]}. Hence the
extension $F_0^h|K_0(x)^h$ must be trivial, so (\ref{as}) holds.
\end{proof}

%
%
\subsection{The case of transcendence degree $>1$}                              
It remains to deduce the second assertion of Theorem~\ref{stt3'} from the first. We will need the following result.
\begin{proposition}                               \label{hrdown}
Take a valued function field $(F|K,v)$ of arbitrary transcendence degree. If for some algebraic extension $K'|K$
the valued function field $(F.K'|K',v)$ is henselian rational, then there is a finite subextension $L|K$ of $K'|K$
such that $(F.L|L,v)$ is henselian rational. The same holds if $L$ is replaced by any of its algebraic extensions.
\end{proposition}
\begin{proof}
Take a transcendence basis $\cT$ of $F.K'|K'$ such that $F.K'\subset K'(\cT)^h$. Since $F|K$ is a function field,
$F$ is generated over $K$ by a finite set $S$ of elements in $F$. As $K'$ is the union over finite extensions $L$
of $K$ contained in $K'$, we also know that $K'(\cT)^h=K'.K(\cT)^h$ is the union over all $L.K(\cT)^h=L(\cT)^h$.
Hence there is a finite subextension $L|K$ of $K'|K$ such that $\cT\subset F.L$ and $S\subset L(\cT)^h$. The
latter implies that $F.L=K(S).L=L(S)\subset L(\cT)^h$. Hence $(F.L|L,v)$ is henselian rational. This
remains true if $L$ is replaced by any of its algebraic extensions, as also the assertions $\cT\subset F.L$ and $S\subset L(\cT)^h$ remain true.
\end{proof}

Note that for $L$ as in the assertion of the proposition, $(F.L|L,v)$ is not necessarily immediate. However, if
$(vF:vK)$ is finite and not divisible by the residue characteristic, and $Fv|Kv$ is finite and separable, then
Hensel's Lemma can be used to show that $L$ can be chosen so that in addition, $(F.L|L,v)$ is immediate.

\parm
Our proof of the second assertion of Theorem~\ref{stt3'} will be done by
induction on the transcendence degree of $F|K$. Assume that $\trdeg F|K=n+1$ with $n\geq 1$ and that the
assertion is proved for every transcendence degree $\leq n$. Take a separating transcendence basis
$\{t_1,\ldots ,t_n,t\}$ of $F|K$.

Assume that $(N,v)$ is a separable immediate extension of $(F,v)$ which is a separably tame field (with
$N|F$ not necessarily being algebraic). Note that $N$ is separable over $K(t_1,\ldots ,t_n,t)$ and hence
also over $K(t_1,\ldots ,t_n)$.
%
%
We take $N'$ to be the relative algebraic closure of $K(t_1,\ldots ,t_n)$ within $N$. As a subextension of
$N|K(t_1,\ldots ,t_n)$, also $N'|K(t_1,\ldots ,t_n)$ is separable, and the same holds for $N'|K$.

Since $(N|K,v)$ and thus also $(N|K(t_1,\ldots ,t_n),v)$ are immediate, it follows from
\cite[Lemma~3.15]{[K7]} that $(N',v)$ is a separably tame field. Further, $(F.N'|N',v)$
is a separable immediate function field of transcendence degree 1. By Proposition~\ref{sttarb},
$(F.N'|N',v)$ is henselian rational.

By Proposition~\ref{hrdown}, there exists a finite extension $L$ of $K(t_1,\ldots ,t_n)$ within $N'$ such that
$(F.L|L,v)$ is henselian rational and the same is true if $L$ is replaced by any of its algebraic extensions.
As the algebraic extension $N'|K(t_1,\ldots ,t_n)$
is separable, the same holds for $N'|L\,$. As a subextension of $N'|K$, also $L|K$ is separable.
Therefore, $(L|K,v)$ is a separable immediate function field of transcendence degree $n$, and
by induction hypothesis it admits a finite extension $L_1$ of $L$ within $N'$ such that
$L_1^h = K(x_1,\ldots ,x_n)^h$ for suitable elements $x_1,\ldots ,x_n\in L_1$. Since also
$(F.L_1|L_1,v)$ is henselian rational, we can write $(F.L_1)^h=L_1(x)^h$ for some $x\in F.L_1\,$.
For $F_1= F.L_1$ it follows that $x_1,\ldots ,x_n,x\in F_1$ and
\[
F_1 = F.L_1 \subset L_1(x)^h = (L_1^h (x))^h = (K(x_1,\ldots ,x_n)^h (x))^h = K(x_1,\ldots ,x_n,x)^h\>,
\]
which shows that $(F_1|K,v)$ is henselian rational.

On the other hand, since $L_1|K(t_1,\ldots ,t_n)$ is finite, $F_1|F$ is a finite subextension of the
immediate separable extension $N|F$. Hence $F_1|F$ is also
separable and immediate. This completes our proof of Theorem~\ref{stt3'}.

%
%
\subsection{Proof of Theorem~\ref{stt3}}           \label{sectprstt3}
Since a tame field is always perfect, the first assertion of Theorem~\ref{stt3'} implies the first assertion of
Theorem~\ref{stt3}. For the proof of the second assertion, we need a slight improvement of Lemma~3.15 of
\cite{[K7]}. We take the occasion to prove a bit more than we will need.
\begin{lemma}
Let $(N,v)$ be a separably tame field and $L\subset N$ a
relatively separable-algebraically closed subfield of $N$. If the residue field
extension $Nv|Lv$ is algebraic, then $(L,v)$ is also a separably
tame field and moreover, $vN/vL$ is torsion free and $Nv = Lv$.
\end{lemma}
\begin{proof}
Denote by $N_0$ the relative algebraic closure of $L$ in $N$. Then also the residue field
extension $Nv|N_0v$ is algebraic. Hence by Lemma~3.15 of \cite{[K7]}, $(N_0,v)$ is a
separably tame field, $vN/vN_0$ is torsion free, and $Nv=N_0v$. By Lemma~3.13 of \cite{[K7]} it follows that
$(N_0^{1/p^{\infty}},v)$ is a tame field. Since the algebraic extension $N_0|L$ is purely inseparable, we have that
$N_0^{1/p^{\infty}}=L^{1/p^{\infty}}$. Using Lemma~3.13 of \cite{[K7]} again, we deduce that $(L,v)$ is a
separably tame field. From the same lemma we also obtain that $(L,v)$ is dense in $(L^{1/p^{\infty}},v)$
and hence also in $(N_0,v)$. Therefore, $vN_0=vL$ and $N_0v=Lv$, showing that $vN/vL$ is torsion free and $Nv=Lv$.
\end{proof}

Now take any immediate function field $(F,v)$ over the tame
field $(K,v)$, and an immediate extension $(N,v)$ of $(F,v)$ which is a tame field. Then in particular, $(N,v)$ is
a separably tame field. Denote by $L$ the relative separable-algebraic closure of $F$ in $N$. Since $(N,v)$ is
immediate over $(F,v)$, it is also immediate over $(L,v)$, and it follows from the above lemma that $(L,v)$ is
a separably tame field. By construction, $(L,v)$ is a separable immediate extension of $(F,v)$, so
Theorem~\ref{stt3'} proves the existence of a finite immediate extension $(F_1,v)$ of $(F,v)$ within $L$ and hence also within $N$ such that $(F_1|K,v)$ is henselian rational.

%
%
\subsection{Proof of Proposition~\ref{N}}
If $\chara Kv=0$, then also $\chara F^h v=0$ and $(F^h,v)$ is a tame field. So it remains to treat the case of
$\chara F^h v=p>0$.
\pars
Assume first that $(K,v)$ is a tame field and $(F|K,v)$ is an immediate extension. By \cite[Theorem~3.2]{[K7]},
$vK = vF$ is $p$-divisible and $Kv = Fv$ is perfect. Let $(N,v)$ be a maximal immediate algebraic extension of
$(F,v)$. Then $(N,v)$ is algebraically maximal, $vN$ is $p$-divisible and $Nv$ is perfect.
Again from \cite[Theorem~3.2]{[K7]} it follows that $(N,v)$ is a tame field.
\pars
Assume now that $(K,v)$ is a separably tame field. In view of what we have proved already, we may assume that
$(K,v)$ is not a tame field. This implies that $\chara K =p>0$.
Assume that $(F|K,v)$ is an immediate separable extension. If $v$ is
trivial on $K$, then $F=K$ since $(F|K,v)$ is immediate. In this case we can just set $N=K$.
So we may assume that $v$ is nontrivial on $K$. Then by \cite[Theorem~3.10]{[K7]}, $vK = vF$ is
$p$-divisible and $Kv = Fv$ is perfect. Let $(N,v)$ be a maximal immediate separable algebraic extension of
$(F,v)$. Then $(N,v)$ is separable-algebraically maximal, $vN$ is $p$-divisible and $Nv$ is perfect.
Again from \cite[Theorem~3.10]{[K7]} it follows that $(N,v)$ is a separably tame field.
\qed

%
%
\section{Proof of Theorem~\ref{stt4}}                \label{sectvalg}
Take a valued function field $(F|K,v)$ such that $vF/vK$ is a torsion
group and $Fv|Kv$ is algebraic. We extend $v$ to the algebraic closure of $F$. The value group $v\tilde{F}$ is
the divisible hull of $vF$ and $v\tilde{K}$ is the divisible hull of $vK$, so they must be equal. Likewise,
$\tilde{F}v$ is the algebraic closure of $Fv$ and $\tilde{K}v$ is the algebraic closure of $Kv$, so they too must
be equal. Therefore, $v\tilde{K}\subseteq v(F.\tilde{K})\subseteq v\tilde{F}=v\tilde{K}$ and
$\tilde{K}v\subseteq v(F.\tilde{K})v\subseteq \tilde{F}v=\tilde{K}v$, so equality holds everywhere.
This shows that $(F.\tilde{K}|\tilde{K},v)$ is an immediate function field and we can apply Theorem~\ref{stt3}.
In particular, if $\trdeg F|K=1$, then $(F.\tilde{K}|\tilde{K},v)$ is henselian rational.
Now an application of Proposition~\ref{hrdown} completes the proof of Theorem~\ref{stt4}.

\end{document}